\documentclass[12pt,a4paper,reqno]{amsart}
\usepackage[english]{babel}
\usepackage[applemac]{inputenc}
\usepackage[T1]{fontenc}
\usepackage{palatino}
\usepackage{amsmath}
\usepackage{amssymb}
\usepackage{amsthm}
\usepackage{amsfonts}
\usepackage{graphicx}
\usepackage{yhmath}
\usepackage[colorlinks = true, citecolor = black]{hyperref}

\title{Constancy results for special families of projections}
\author{Katrin F\"assler and Tuomas Orponen}\email{katrin.fassler@helsinki.fi \\ tuomas.orponen@helsinki.fi}
\thanks{K. F\"{a}ssler was supported by the Swiss National Science Foundation. T. Orponen acknowledges the financial support of the Finnish National Graduate School in Mathematics and its Applications.}

\newcommand{\R}{\mathbb{R}}
\newcommand{\N}{\mathbb{N}}

\newcommand{\Z}{\mathbb{Z}}
\newcommand{\Hd}{\dim_{\mathrm{H}}}

\newcommand{\calD}{\mathcal{D}}

\newcommand{\calE}{\mathcal{E}}
\newcommand{\calT}{\mathcal{T}}

\newcommand{\calH}{\mathcal{H}}

\newcommand{\spt}{\operatorname{spt}}
\newcommand{\card}{\operatorname{card}}

\newcommand{\p}{\mathbf{p}}
\newcommand{\B}{\operatorname{B}}
\newcommand{\MB}{\operatorname{MB}}
\newcommand{\Dim}{\operatorname{Dim}}

\newcommand{\m}{\mathfrak{m}}
\newcommand{\dist}{\operatorname{dist}}

\numberwithin{equation}{section}

\theoremstyle{plain}
\newtheorem{thm}[equation]{Theorem}
\newtheorem{lemma}[equation]{Lemma}

\newtheorem{cor}[equation]{Corollary}
\newtheorem{proposition}[equation]{Proposition}

\theoremstyle{definition}
\newtheorem{definition}[equation]{Definition}

\theoremstyle{remark}
\newtheorem{remark}[equation]{Remark}

\begin{document}

\begin{abstract} Let $\{\mathbb{V}=V\times \mathbb{R}^l : V \in G(n-l,m-l)\}$ be the family of $m$-dimensional subspaces of $\R^{n}$ containing $\{0\}\times \mathbb{R}^l$, and let $\pi_{\mathbb{V}}\colon \R^n \to \mathbb{V}$ be the orthogonal projection onto $\mathbb{V}$. We prove that the mapping $V \mapsto \Dim \pi_{\mathbb{V}}(B)$ is almost surely constant for any analytic set $B \subset \mathbb{R}^n$, where $\Dim$ denotes either Hausdorff or packing dimension.
\end{abstract}

\maketitle

\section{Introduction}

Unlike Hausdorff dimension, packing dimension is not generally preserved by orthogonal projections. In 1994, M. J\"{a}rvenp\"{a}\"{a}  exhibited in her PhD thesis \cite{Jar} a compact set $K \subset \R^{n}$ of packing dimension $\dim_{\p} K =: s \leq m$, such that the projections of $K$ onto every $m$-dimensional subspace in $\R^{n}$ have packing dimension strictly smaller than $s$. Three years later, K. Falconer and J. Howroyd \cite{FH} discovered a curious phenomenon: the packing dimension of the projections is almost surely constant -- only this constant need not be $s$.

In the present paper, we aim for similar results in a context different from Falconer and Howroyd's. We consider some (particular) subfamilies of the family of all orthogonal projections from $\mathbb{R}^n$ to $m$-dimensional subspaces -- the simplest case covered being the projections onto all `vertical' planes in $\R^{3}$. It is obvious that, in general, these subfamilies of projections preserve neither Hausdorff- nor packing dimension. We address the constancy questions in Theorems \ref{main2_higher} and \ref{main1_higher} by proving that `maximal behavior is typical behavior' for the dimension of projections. Such results do not follow from the classical projection theorems of Mar{\-}strand, Kaufman and Mattila, even if one takes into account the refined versions with exceptional sets, see \cite{Ka}. We should also mention that our techniques are quite different from the ones developed in \cite{FH}.

As far as we know, constancy issues have not been studied previously for `small' families of projections, that is, families with a parameter set smaller than the whole Grassmannian $G(n,m)$. However, such families have received some attention quite recently. In \cite{JJK}, concluding the work started in \cite{JJLL}, E. J\"arvenp\"a\"a, M. J\"arvenp\"a\"a and T. Keleti provide a complete answer to the following question: given a general `small' non-degenerate family of projections in $\R^{n}$, how much can the dimension of a set $B \subset \R^{n}$ (or a measure) drop under these projections? We emphasize that our families of projections are nowhere as general as the ones studied in \cite{JJK}. The reason is simple: it is not clear to us, what is the greatest generality in which constancy results -- such as the ones below -- can be proven. At any rate, they are not true for all families considered in \cite{JJK}: for instance, it is easy to find one-parameter families of projections onto planes in $\R^{3}$, for which there is no hope of constancy of any kind.

It is time to introduce the particular families of projections we will be concerned with. They are projections onto $m$-planes in $\R^{n}$, $2 \leq m < n$, parameterized by the Grassmannian $G(n-l,m-l)$ for some $0<l<m$. Since
\begin{displaymath}
 \Hd G(n-l,m-l)= (m-l)(n-m)<m(n-m)=\Hd G(n,m),
\end{displaymath}
such families are `small' in the sense introduced above. Write $\mathbb{V}=V\times \R^{l}$ for the $m$-dimensional subspace of $\R^{n}$ containing $\{0\}\times \R^l$, where $V$ is an element of $G(n-l,m-l)$. We are interested (only) in the orthogonal projections $\pi_{\mathbb{V}}\colon \R^n\to \mathbb{V}$. The simplest case is obtained with $n = 3$, $m = 2$ and $l = 1$: then the mappings $\pi_{\mathbb{V}}$ are the orthogonal projections onto the `vertical' planes in $\R^{3}$, that is, the planes containing the $z$-axis $\{0\} \times \R$.

We write $B_{\mathbb{V}}=\pi_{\mathbb{V}}(B)$. We will also make use of the projections onto the $(m-l)$-dimensional subspaces $V$; we will denote by $\pi_V$ both the orthogonal projection $\R^{n-l} \to V$ and $\R^{n} \to V\times \{0\}$. We write $B_{V} = \pi_V (B)$, and denote by $\gamma_{n,m}$ the natural $O(n)$ invariant measure on the Grassmanian $G(n,m)$, see \cite[3.9]{Mat2}. Furthermore, $G(n,m)$ will be endowed with the metric $d_{\pi}$
given by
\begin{equation}\label{eq:dist}
d_{\pi}(V,W)=\|\pi_V-\pi_W\|,\quad V,W\in G(n,m),
\end{equation}
where $\|\cdot\|$ denotes the operator norm. Below and above, $\Hd$ refers to Hausdorff dimension, whereas $\dim_{\p}$ refers to packing dimension and $\overline{\dim}_{\B}$ denotes the upper box dimension.

Before stating our main results on constancy, let us observe as Proposition \ref{p:hausdorff} that for sets $B$ with small enough dimension, it is  possible to give an almost sure formula for $\Hd B_{\mathbb{V}}$ in terms of $\Hd B$. Perhaps surprisingly, the proposition is not a corollary of the bounds in \cite{JJK}, as they are not sharp for our particular families of projections. This only testifies that our projections have a very special form -- and the proof of Proposition \ref{p:hausdorff} heavily relies on this fact. Let us clarify the point with an example: according to the proposition below, the Hausdorff dimension of every $1$-dimensional set is almost surely preserved under the $2$-dimensional family of projections onto $2$-dimensional planes in $\mathbb{R}^4$ which contain $\{0\}\times \mathbb{R}$. However, it is not true that the dimension of such sets is preserved under \emph{arbitrary} non-degenerate $2$-dimensional families of projections from $\mathbb{R}^4$ to $2$-dimensional planes. Consider, for instance, the family of projections associated to $2$-planes contained in $\mathbb{R}^3 \times \{0\}$. For these projections, the set $B=\{0\}\times \mathbb{R}$ is projected to a point for all considered directions, and so the dimension can drop from one to zero.

\begin{proposition}\label{p:hausdorff}
Let $B \subset \R^{n}$ be an analytic set with $\Hd B \leq m-l$. Then $\Hd B_{\mathbb{V}}= \Hd B$ for $\gamma_{n-l,m-l}$ almost every $V \in G(n-l,m-l)$.
\end{proposition}

For sets $B$ of dimension bigger than $m-l$, it is no longer possible to give an almost sure formula for $\Hd B_{\mathbb{V}}$ in terms of $\Hd B$. Instead, we have the following constancy results.

\begin{thm}\label{main2_higher} Let $B \subset \R^{n}$ be an analytic set, and write
\begin{displaymath} \m_{\mathrm{H}} := \sup\{\Hd B_{\mathbb{V}} : V \in G(n-l,m-l)\}. \end{displaymath}
Then, the $\Hd B_{\mathbb{V}} = \m_{\mathrm{H}}$ for $\gamma_{n-l,m-l}$ almost every $V\in G(n-l,m-l)$.
\end{thm}

\begin{thm}\label{main1_higher} Let $B \subset \R^{n}$ be a bounded analytic set. Write
\begin{displaymath} \m_{\B} := \sup \{\overline{\dim}_{\B} B_{\mathbb{V}} : V \in G(n-l,m-l)\}, \; \m_{\p} := \sup\{\dim_{\p} B_{\mathbb{V}} : V \in G(n-l,m-l)\}. \end{displaymath}
Then, the sets
\begin{displaymath} E_{\B} := \{V \in G(n-l,m-l) : \overline{\dim}_{\B} B_{\mathbb{V}} \neq \m_{\B} \},  \end{displaymath}
\begin{displaymath}
E_{\p} := \{V \in G(n-l,m-l) : \dim_{\p} B_{\mathbb{V}} \neq \m_{\p}\}
\end{displaymath}
are meagre and have $\gamma_{n-l,m-l}$ measure zero. The statement concerning packing dimension holds for unbounded sets as well.
\end{thm}


\begin{remark}\label{bounded} {A routine argument shows that it is sufficient to prove Theorems \ref{main2_higher} and \ref{main1_higher} for bounded sets. Thus, we may and will only consider bounded sets in the sequel.}
\end{remark}

Throughout the paper we write $a \lesssim b$, if $a \leq Cb$ for some constant $C \geq 1$. Should we wish to emphasize that $C$ depends on some parameter $p$, we may write $a \lesssim_{p} b$. The chain $a \lesssim b \lesssim a$ is abbreviated to $a \asymp b$. If $d$ is a metric on a space $X$, we denote by $B_{d}(x,r)$ the closed ball with center $x \in X$ and radius $r$; the subscript $d$ is dropped, if the metric is obvious from the context.

\section{Acknowledgements}

We are grateful to Pertti Mattila for suggesting the problem, for stimulating discussions and for pointing out the reference \cite{Wh}. We also wish to thank the referee for his/her quick yet exceptionally careful review of the manuscript.

\section{Proof for the Hausdorff dimension}

\begin{proof}[Proof of Proposition \ref{p:hausdorff}]
Let $0<s<t<\Hd B \leq m-l$. By Frostman's lemma there exists a non-trivial finite measure $\mu$ with support in $B$, which satisfies the growth condition $\mu(B(x,r))\leq r^t$ for $x\in \mathbb{R}^n$ and $r>0$. It follows that the associated $s$-energy is finite,
\begin{displaymath}
I_s(\mu) := \int \int |x-y|^{-s}d\mu(x)d\mu(y)<\infty.
\end{displaymath}
By the definition of the push-forward measure $\pi_{\mathbb{V}\sharp}\mu$ and Fubini's theorem,
\begin{align*}
\int_{G(n-l,m-l)}&I_s(\pi_{\mathbb{V}\sharp}\mu)\,d \gamma_{n-l,m-l}(V)\\&= \int_B \int_B \int_{G(n-l,m-l)} |\pi_{\mathbb{V}}(x)-\pi_{\mathbb{V}}(y)|^{-s}\, d \gamma_{n-l,m-l}(V) d \mu(x) d \mu(y).
\end{align*}
Now if $(x,y)\in B\times B$, $x\neq y$, is such that
\begin{displaymath}
\sum_{i=n-l+1}^{n} (x_i-y_i)^2 \leq \sum_{i=1}^{n-l}(x_i-y_i)^2,
\end{displaymath}
we have
\begin{align*}
\int_{G(n-l,m-l)}|\pi_{\mathbb{V}}(x-y)|^{-s}\,d \gamma_{n-l,m-l}(V)&\leq \int_{G(n-l,m-l)}|\pi_{V}(\pi(x)-\pi(y))|^{-s}\,d \gamma_{n-l,m-l}(V),\\
&\lesssim |\pi(x)-\pi(y)|^{-s}\asymp|x-y|^{-s}
\end{align*}
where $\pi: \R^n \to \R^{n-l}$ is defined by $\pi(x)=(x_1,\ldots,x_{n-l})$. Here the second inequality follows from Corollary 3.12 in \cite{Mat2}.

On the other hand, if $(x,y)\in B \times B$ is such that
\begin{displaymath}
\sum_{i=1}^{n-l}(x_i-y_i)^2 < \sum_{i=n-l+1}^{n} (x_i-y_i)^2,
\end{displaymath}
the pointwise estimate
\begin{displaymath}
|\pi_{\mathbb{V}}(x)-\pi_{\mathbb{V}}(y)|^{-s}\leq \left(\sum_{i=n-l+1}^n (x_i-y_i)^2\right)^{-s/2} \asymp |x-y|^{-s}.
\end{displaymath}
holds. Together, these observations yield
\begin{displaymath}
\int_{G(n-l,m-l)} I_s(\pi_{\mathbb{V}\sharp}\mu)\,d \gamma_{n-l,m-l}(V) \lesssim \int_B \int_B |x-y|^{-s}\,d\mu(x) d \mu(y) = I_s(\mu)<\infty
\end{displaymath}
and thus $\Hd B_{\mathbb{V}}\geq s$ for almost every $V\in G(n-l,m-l)$. The result follows.
\end{proof}

{The method of bounding energy integrals used in the proof of Proposition \ref{p:hausdorff} cannot be applied to derive information on the Hausdorff dimension of projections of sets of dimension $s > m - l$; the problem is that integrals of the form
\begin{displaymath} \int_{G(n-l,m-l)}|\pi_V(x)-\pi_V(y)|^{-s}\,d \gamma_{n-l,m-l}(V) \end{displaymath}
can be infinite in that case, which means that the average of the energies $I_{s}(\pi_{\mathbb{V}\sharp}\mu)$ over the planes $\mathbb{V}$ may easily be infinite as well. This is natural, recalling that the projections of $\spt \mu$ can have dimension strictly smaller than $\spt \mu$ for all planes $\mathbb{V}$. Consequently, we need to devise a new quantity to replace $I_{s}(\pi_{\mathbb{V}\sharp}\mu)$, which (a) is bounded in size so that that it integrates over the planes $\mathbb{V}$, yet (b) contains all vital information on the dimension of $\spt \pi_{\mathbb{V}\sharp}\mu$. The trick is to discretise $\mu$ on a scale $\delta > 0$, thus turning $\mu$ into an $L^{2}$-function $\mu_{\delta}$. Then, projecting $\mu_{\delta}$ -- instead of $\mu$ -- onto the planes $\mathbb{V}$ results in a family of $L^{2}$-functions, denoted by $(\mu_{\delta})_{\mathbb{V}}$. It turns out that the $L^{2}$-norms $\|(\mu_{\delta})_{\mathbb{V}}\|_{2}$, for various $\delta > 0$, provide a substitute for $I_{s}(\pi_{\mathbb{V}\sharp}\mu)$ satisfying both requirements (a) and (b).}

{In contrast with many classical proofs related to projection phenomena, the measure $\mu$ we consider is not simply a Frostman measure supported on the set $B \subset \R^{n}$ we are projecting. Rather, $\mu$ is an abstract pull-back of a Frostman measure $\nu$ supported on one of the projections of $B$, namely the one with the (essentially) largest dimension. The key observation in the proof is that if the norms $\|\nu_{\delta}\|_{2}$ satisfy certain growth estimates, then the same estimates automatically transfer to the norms $\|(\mu_{\delta})_{\mathbb{V}}\|_{2}$, for almost all planes $\mathbb{V}$. Having related these growth estimates to the dimensions of $\spt (\mu_{\delta})_{\mathbb{V}}$, this translates into our claim that almost all projections of $\spt \mu$ have dimension at least $\dim \spt \nu$.}

{In our first lemma, we make precise the idea of discretising a measure $\mu$ on a scale $\delta > 0$, and relate the growth rate of $\|\mu_{\delta}\|_{2}$, as $\delta \searrow 0$, to the dimension of $\spt \mu$.}



\begin{lemma}\label{L1_higher} Let $\mu$ be a finite measure on $\R^{m}$, and let $(\psi_{\delta_j})_{j\in\mathbb{N}}$ be a collection of smooth functions of the form
\begin{displaymath} \psi_{\delta_j}(x) = \delta_j^{-m}\psi(x/\delta_j), \end{displaymath}
where $\psi$ is a fixed non-negative compactly supported smooth function, not equal to zero, and $\delta_j=2^{-j}$. Suppose that the growth of the $L^{2}$-norms of the convolutions $\mu_{\delta_j} := \mu \ast \psi_{\delta_j}$ is bounded as follows:
\begin{equation}\label{growth} \|\mu_{\delta_j}\|_{2}^{2} \lesssim \delta_j^{s - m}\quad\text{for some }0<s<m.\end{equation}
Then $\Hd \spt \mu \geq s$.
\end{lemma}

\begin{proof} Parseval's theorem and \eqref{growth} give
\begin{displaymath} \int_{\R^{m}} |\hat{\mu}(x)|^{2}|\hat{\psi}(\delta_j x)|^{2} \, dx = \int_{\R^{m}} |\widehat{\mu \ast \psi_{\delta_j}}(x)|^{2} \, dx = \|\mu_{\delta_j}\|_{2}^{2} \lesssim \delta_j^{s - m}. \end{displaymath}
Next, observe that there exists a constant $c > 0$ such that $|\hat{\psi}(\delta_j x)|^{2} \geq c$ for $|x| \leq c\delta_j^{-1}$. Let $0<r < s$. The $r$-energy of $\mu$ can be expressed through the Fourier transform $\hat \mu$, see for instance \cite[Lemma 12.12]{Mat2}.
Then,
\begin{align*} I_r(\mu) \asymp \int_{\R^{m}} |\hat{\mu}(x)|^{2}|x|^{r - m} \, dx & \lesssim 1 + \sum_{j = 1}^{\infty} 2^{j(r - m)} \int_{B(0,c2^{j})} |\hat{\mu}(x)|^{2} \, dx\\
& \lesssim 1 + \sum_{j = 1}^{\infty} 2^{j(r - s)} 2^{j(s - m)} \int |\hat{\mu}(x)|^{2}|\hat{\psi}(2^{-j} x)|^{2} \, dx\\
& \lesssim 1 + \sum_{j = 1}^{\infty} 2^{j(r - s)} < \infty, \end{align*}
which means that $I_{r}(\mu) < \infty$, and so $\Hd \spt \mu \geq r$.
\end{proof}

\begin{lemma}\label{comparison_higher} Let $\calD_{\delta}$ be a partition of $\R^{d}$ into dyadic cubes of side-length $\delta$; thus, $\calD_{1} := \{\Pi_{i=1}^d [m_i,m_i+1) : m_i \in \Z\}$, and $\calD_{\delta} := \{\delta Q : Q \in \calD_{1}\}$. Suppose that $\nu$ is a measure on $\R^{d}$ of the form
\begin{displaymath} \nu = \sum_{Q \in \calD_{\delta}} c_{Q}\mathcal{L}^d\llcorner_{Q}, \qquad c_{Q} \geq 0, \end{displaymath}
where $\mathcal{L}^d\llcorner_{Q}$ denotes the restriction of the Lebesgue measure to $Q$.
Then $$I_{s}(\nu) \lesssim \delta^{t - s}I_{t}(\nu)$$ for all $t,s$ with  $0 < t \leq s < d$. The implicit constants depend on $d,s$ and $t$, but not on $\delta > 0$ or the particular choice of $\nu$, as long as it is of the form indicated above.
\end{lemma}

\begin{proof} Define the relation $\sim$ on $\calD_{\delta} \times \calD_{\delta}$ by
\begin{displaymath} Q \sim Q' \qquad \Longleftrightarrow \qquad \overline{Q} \cap \overline{Q'} \neq \emptyset. \end{displaymath}
If $x \in \R^{d}$, let $Q_{x} \in \calD_{\delta}$ be the unique cube containing $x$. For $x,y \in \R^{n}$, we write $x \sim y$, if $Q_{x} \sim Q_{y}$. Then,
\begin{displaymath} I_{s}(\nu) = \iint_{\{(x,y) : x \sim y\}} |x - y|^{-s} \, d\nu x d\nu y + \iint_{\{(x,y) : x \not\sim y\}} |x - y|^{-s} \, d\nu x d\nu y,
\end{displaymath}
For the second term, it suffices to note that $x \not\sim y$ implies $|x - y| \geq \delta$, whence $|x - y|^{-s} \leq \delta^{t - s}|x - y|^{-t}$. To estimate the first term, write
\begin{align*} \iint_{\{(x,y) : x \sim y\}} |x - y|^{-s} \, d\nu x d\nu y & = \sum_{Q \sim Q'} c_{Q}c_{Q'} \int_{Q} \int_{Q'} |x - y|^{-s} \, dx dy\notag\\
& \leq \sum_{Q \sim Q'} c_{Q}c_{Q'} \int_{Q} \int_{B(y,c(d)\delta)} |x - y|^{-s} \, dx dy\\
& \asymp \delta^{2d-s} \sum_{Q \sim Q'} c_{Q}c_{Q'}, \end{align*}
where the constant $c(d)$ is chosen large enough so that for $Q'\sim Q$ and $y\in Q$, we have $Q' \subseteq B(y, c(d)\delta)$. Here we have used
\begin{align*}
\int_{B(y,c(d)\delta)} |x-y|^{-s}\,dx&= \int_{B(0,c(d)\delta)}|x|^{-s}\,d x= \int_0^{c(d)\delta}\int_{S^{d-1}}r^{-s}r^{d-1}\,d\sigma^{d-1} dr\\&= \frac{(c(d)\delta)^{d-s}}{d-s} \int_{S^{d-1}}d\sigma^{d-1}
\end{align*}
and $\int_Q \, dy = \delta^d$.

To bound the sum $\sum_{Q \sim Q'} c_Q c_{Q'}$, note that if $Q \in \calD_{\delta}$ is fixed, it has only a  finite number $N(d)$ of `neighbours' $Q' \in \calD_{\delta}$. In particular,
\begin{displaymath}
\sum_{Q\in\mathcal{D}_{\delta}} \left(\max_{Q\sim Q'}c_{Q'}\right)^2\leq N(d)\sum_{Q\in\mathcal{D}_{\delta}}c_Q^2
\end{displaymath}
and thus,
\begin{displaymath} \sum_{Q \sim Q'} c_{Q}c_{Q'} = \sum_{Q \in \calD_{\delta}} c_{Q} \sum_{Q' \in \calD_{\delta} : Q' \sim Q} c_{Q'} \lesssim \sum_{Q \in \calD_{\delta}} \left( c_{Q} \cdot \max_{Q' \sim Q} c_{Q'} \right)  \lesssim \sum_{Q \in \calD_{\delta}} c_{Q}^{2}, \end{displaymath}
which implies that
\begin{equation}\label{form6} I_{s}(\nu) \lesssim \delta^{2d-s}\sum_{Q \in \calD_{\delta}} c_{Q}^{2} + \delta^{t - s}I_{t}(\nu). \end{equation}
Let us next bound the $t$-energy $I_{t}(\nu)$ from below. If $Q \in \calD_{\delta}$, let $Q^{o}$ be the cube which is concentric with $Q$ but has only half the side-length. Then, if $x \in Q^{o}$, we have $B(x,\delta/c(d)) \subset Q$ for large enough $c(d)$, not necessarily the same as above, and this shows that
\begin{align*} I_{t}(\nu) \geq \sum_{Q \in \calD_{\delta}} c_{Q}^{2} \int_{Q^{o}} \int_{B(y,\delta/c(d))} |x - y|^{-t} \, dx dy \asymp \delta^{2d - t}\sum_{Q \in \calD_{\delta}} c_{Q}^{2},  \end{align*}
by a similar integration in spherical coordinates as before.
It now follows from \eqref{form6} that
\begin{displaymath} I_{s}(\nu) \lesssim \delta^{t - s} \left( \delta^{2d - t}\sum_{Q \in \calD_{\delta}} c_{Q}^{2} \right) + \delta^{t - s}I_{t}(\nu) \lesssim \delta^{t - s}I_{t}(\nu), \end{displaymath}
as claimed. \end{proof}

\begin{proof}[Proof of Theorem \ref{main2_higher}] Let $B \subset \R^{n}$ be an analytic set. Recall that
\begin{displaymath} \m_{\mathrm{H}} = \sup\{\Hd B_{V \times \R^l} : V \in G(n-l,m-l)\}\leq m, \end{displaymath}
and we intend to prove that $\Hd B_{\mathbb{V}} = \m_{\mathrm{H}}$ almost surely. To this end, we may assume that $\m_{\mathrm{H}} > 0$. Let $0 < \sigma < \m_{\mathrm{H}}$ and find a subspace $V_0 \in G(n-l,m-l)$ such that $\Hd B_{V_0\times \mathbb{R}^l} > \sigma$. We will identify all the subspaces $\mathbb{V} = V \times \R^l$ with $\R^{m}$, so that $B_{\mathbb{V}} \subset \R^{m}$, and the projections $\pi_{\mathbb{V}}=\pi_{V \times \R^l}$, $V \in G(n-l,m-l)$, will all be $\R^{m}$-valued. Let $\Psi$ be a non-negative radial symmetric smooth function on $\R^{n}$, satisfying
\begin{equation}\label{form10_higher} \chi_{B(0,1)} \leq \Psi \leq \chi_{B(0,2)}. \end{equation}
Then, for any $V \in G(n-l,m-l)$, the projection $\Psi_{\mathbb{V}}$ of $\Psi$ to $\R^{m}$, defined by
\begin{displaymath} \Psi_{\mathbb{V}}(x) = \int_{\pi_{\mathbb{V}}^{-1}\{x\}} \Psi \, d\calH^{n-m}, \end{displaymath}
is a non-negative compactly supported smooth function on $\R^{m}$, not identically equal to zero. Since $\Psi$ is radial symmetric, the projections $\Psi_{\mathbb{V}}$ are independent of $\mathbb{V}$; to emphasise this, we write $\psi := \Psi_{\mathbb{V}}$. The plan of the proof is to use Lemma \ref{L1_higher} as follows. We will find a finite Borel measure $\mu$ supported on the analytic $B$ such that the growth estimate
\begin{equation}\label{form7_higher} \|(\mu_{{V} \times \R^l})_{\delta_j}\|_{2}^{2} \lesssim \delta_j^{s - m}, \qquad 0 \leq s < \sigma, \end{equation}
holds for $\gamma_{n-l,m-l}$ almost every $V \in G(n-l,m-l)$, where $\mu_{V \times \R^{l}} = \mu_{\mathbb{V}}= \pi_{\mathbb{V}\sharp} \mu$, and $\delta_{j} = 2^{-j}$.
As in Lemma \ref{L1_higher}, the measure $(\mu_{\mathbb{V}})_{\delta_j}$ is defined to be the convolution $\mu_{\mathbb{V}} \ast \psi_{\delta_j}$, where $\psi_{\delta_j}(x) = \delta_j^{-m}\psi(x/\delta_j)$. According to Lemma \ref{L1_higher}, establishing \eqref{form7_higher} will complete the proof of Theorem \ref{main2_higher}.

Before defining the measure $\mu$, let us make one observation to simplify the proof of \eqref{form7_higher}.

{
\begin{lemma}
Let $\mu$ be a finite Borel measure on $\mathbb{R}^n$. Then for all $\delta>0$ and $\mathbb{V}=V\times \mathbb{R}^l$, $V\in G(n-l,m-l)$, we have
\begin{equation}\label{form8_higher} (\mu_{\mathbb{V}})_{\delta} = (\mu_{\delta})_{\mathbb{V}}, \end{equation}
where $\mu_{\delta} := \mu \ast \Psi_{\delta}$ and $(\mu_{\mathbb{V}})_{\delta}= \mu_{\mathbb{V}}\ast \psi_{\delta}$ with $\psi=\Psi_{\mathbb{V}}$.
\end{lemma}
\begin{proof}
Writing $\Psi_{\delta}(x) := \delta^{-n}\Psi(x/\delta)$, and denoting the transpose of the projection $\pi_{\mathbb{V}}$ by $\pi_{\mathbb{V}}^{\mathrm{T}} \colon \R^{m} \to \R^{n}$, we have the identity
\begin{align*} \widehat{(\mu_{\mathbb{V}})_{\delta}}(x) & = \widehat{\mu_{\mathbb{V}} \ast \psi_{\delta}}(x) = \widehat{\mu_{\mathbb{V}}}(x)\widehat{\psi_{\delta}}(x)\\
& = \hat{\mu}(\pi_{\mathbb{V}}^{\mathrm{T}}(x))\widehat{\Psi}(\pi_{\mathbb{V}}^{\mathrm{T}}(\delta x)) = \hat{\mu}(\pi_{\mathbb{V}}^{\mathrm{T}}(x))\widehat{\Psi}(\delta \cdot \pi_{\mathbb{V}}^{\mathrm{T}}(x))\\
& = \hat{\mu}(\pi_{\mathbb{V}}^{\mathrm{T}}(x))\widehat{\Psi_{\delta}}(\pi_{\mathbb{V}}^{\mathrm{T}}(x)) = \widehat{\mu \ast \Psi_{\delta}}(\pi_{\mathbb{V}}^{\mathrm{T}}(x)) = \widehat{(\mu \ast \Psi_{\delta})_{\mathbb{V}}}(x)  \end{align*}
for all $x \in \R^{m}$.
\end{proof}}

So, the order of discretising and projecting can be interchanged, and, in particular, $\|(\mu_{\mathbb{V}})_{\delta}\|_{2} = \|(\mu_{\delta})_{\mathbb{V}}\|_{2}$ for all $V \in G(n-l,m-l)$ and $\delta > 0$. But, in order to apply Lemma \ref{comparison_higher}, we will need something more. Let $\calD_{\delta}$ be the collection of dyadic cubes of side-length $\delta > 0$ in $\R^{n}$, as defined above for $d=n$. If $\mu$ is any finite Borel measure on $\R^{n}$, set
\begin{displaymath} \mu^{\delta} := \sum_{Q \in \calD_{\delta}} \frac{\mu(Q)}{\delta^{n}}\chi_{Q}. \end{displaymath}
Eventually, we will control the $L^{2}$-norms in \eqref{form7_higher} by estimating the $L^{2}$-norms of the projections $(\mu^{\delta})_{\mathbb{V}}$. This is reasonable thanks to the following lemma.
\begin{lemma} If $\mu$ is any finite Borel measure on $\R^{n}$, we have
\begin{equation}\label{form9_higher} \|(\mu_{\mathbb{V}})_{\delta}\|_{2} \lesssim \|(\mu^{\delta})_{\mathbb{V}}\|_{2} \lesssim \|(\mu_{\mathbb{V}})_{c\delta}\|_{2}\end{equation}
for any $\delta > 0$ and $V \in G(n-l,m-l)$. The implicit constants in \eqref{form9_higher} only depend on $n$ and the choice of the function $\Psi$, as in \eqref{form10_higher}, and the constant $c$ depends only on the dimension $n$.
\end{lemma}

\begin{proof} Fix $V \in G(n-l,m-l)$ and $\delta > 0$. Using \eqref{form10_higher}, we first make an estimate in $\R^{n}$:
\begin{align*}  \mu_{\delta}(x) & {\lesssim} \delta^{-n} \int_{B(x,2\delta)} \, d\mu y \leq \delta^{-n} \mathop{\sum_{Q \in \calD_{\delta}}}_{Q \cap B(x,2\delta) \neq \emptyset} \mu(Q)\\
& \leq \delta^{-n} \int_{B(x,c(n)\delta)} \sum_{Q \in \calD_{\delta}} \frac{\mu(Q)}{\delta^{n}}\chi_{Q}(y) \, dy {\lesssim} (\mu^{\delta})_{c(n)\delta}(x),  \end{align*}
where $c(n)$ is large enough so that $Q\cap B(x,2\delta) \neq \emptyset$ implies $Q\subset B(x,c(n)\delta)$.
Here $(\mu^{\delta})_{c(n)\delta} = \mu^{\delta} \ast \Psi_{c(n)\delta}$, as usual. Applying the previous estimate and \eqref{form8_higher} twice we obtain
\begin{displaymath} (\mu_{\mathbb{V}})_{\delta} = (\mu_{\delta})_{\mathbb{V}} \lesssim ((\mu^{\delta})_{c(n)\delta})_{\mathbb{V}} = ((\mu^{\delta})_{\mathbb{V}})_{c(n)\delta}, \end{displaymath}
where $((\mu^{\delta})_{\mathbb{V}})_{c(n)\delta} = (\mu^{\delta})_{\mathbb{V}} \ast \psi_{c(n)\delta}$, as before.
 Now it suffices to note that the convolution of any function $f \in L^{1}(\R^{m})$ with $\psi_{\delta}$ is controlled by a constant times the Hardy-Littlewood maximal function $Mf$, and the constant can be chosen to be independent of $f$. Indeed,
 \begin{align*}
 |f\ast \psi_{\delta}(x)|&\leq \int_{\R^m} |\psi_{\delta}(x-y)||f(y)|dy\leq \frac{1}{\delta^m} \int_{B(x,2\delta)}|f(y)|dy \\&\asymp \frac{1}{\mathcal{L}^m(B(x,2\delta))}\int_{B(x,2\delta)}|f(y)|dy\leq Mf(x).
 \end{align*}
Applying this to $(\mu^{\delta})_{\mathbb{V}}$, it follows that
\begin{displaymath} \|(\mu_{\mathbb{V}})_{\delta}\|_{2} \lesssim \|((\mu^{\delta})_{\mathbb{V}})_{c(n)\delta}\|_2 \lesssim \|M(\mu^{\delta})_{\mathbb{V}}\|_{2} \lesssim \|(\mu^{\delta})_{\mathbb{V}}\|_{2}, \end{displaymath}
where the upper bound is simply the boundedness of operator $M$ in $L^2$.

For the converse inequality let us observe that \eqref{form10_higher} guarantees the existence of a constant $c$ which depends only on the dimension $n$ such that
\begin{displaymath} c^{-n} \mu^{\delta}(x)\leq (c\delta)^{-n} \mu(B(x,c\delta))\leq \mu_{c\delta}(x) ,
\end{displaymath}
so that $\|(\mu^{\delta})_{\mathbb{V}}\|_2 \lesssim \|(\mu_{c\delta})_{\mathbb{V}}\|_2 = \|(\mu_{\mathbb{V}})_{c\delta}\|_2$ as desired.
\end{proof}

Next, we will define the measure $\mu$, for which \eqref{form7_higher} will be verified. At the beginning of the proof, we found a special subspace $V_0 \in G(n-l,m-l)$ such that $\Hd B_{\mathbb{V}_0} > \sigma$. Since $B_{\mathbb{V}_0} \subset \R^{m}$ is an analytic set, we may use Frostman's lemma to find a non-trivial finite Borel measure $\mu_{\mathbb{V}_0}$, supported on $B_{\mathbb{V}_0}$ and satisfying $I_{\sigma}(\mu_{\mathbb{V}_0}) < \infty$. We may then `pull back' the measure $\mu_{\mathbb{V}_0}$ inside the set $B$ using the following result of A. Lubin from 1974.
\begin{lemma}[Corollary 6 in \cite{Lu}]\label{LuLemma} Let $X,Y$ be analytic subsets of complete separable metric spaces, and let $f \colon X \to Y$ be a Borel function. Then, if $\nu$ is a measure supported on $f(X) \subset Y$, there exists a Borel measure $\mu$ on $X$ such that $f_{\sharp}\mu = \nu$.
\end{lemma}
We apply the lemma with $X = B$, $Y = B_{\mathbb{V}_0}$ and $\nu = \mu_{\mathbb{V}_0}$ to obtain a measure $\mu$, supported on $B$, and such that
\begin{displaymath} \pi_{\mathbb{V}_0\sharp}\mu = \mu_{\mathbb{V}_0}. \end{displaymath}
For $V \in G(n-l,m-l)$, we write $\mu_{\mathbb{V}} := \pi_{V \times \mathbb{R}^l\sharp}\mu$; clearly, the two definitions of $\mu_{\mathbb{V}_0}$ coincide. For $V = V_0$, we have the estimate
\begin{align*}\|(\mu^{\delta})_{\mathbb{V}_0}\|_{2}^{2} \lesssim \|(\mu_{\mathbb{V}_0})_{c\delta}\|_{2}^{2} & \asymp \int_{\R^{m}} |\widehat{\mu_{\mathbb{V}_0}}(x)|^{2}|\hat{\psi}(c\delta x)|^{2} \, dx\\
& \lesssim (c\delta)^{\sigma - m} \int_{\R^{m}} |\widehat{\mu_{\mathbb{V}_0}}(x)|^{2}|x|^{\sigma - m}\,dx \lesssim \delta^{\sigma - m}, \end{align*}
using the rapid decay bound $|\widehat{\psi}(y)| \lesssim |y|^{(\sigma - m)/2}$ for $y \in \R^{m}$ (note that $\psi$ is a Schwartz function) and the finiteness of the $\sigma$-energy of $\mu_{\mathbb{V}_0}$.

So, we have \eqref{form7_higher} for $V = V_0$. Using \textbf{only this information}, we intend to prove \eqref{form7_higher} for $\gamma_{n-l,m-l}$ almost all directions $V \in G(n-l,m-l)$.

Given $h \in \R^l$, let $\mu^{\delta}_{h} \colon \R^{n-l} \to [0,\infty)$ be the function $\mu^{\delta}_{h}(x) = \mu^{\delta}(x,h)$. Recalling the definition of $\mu^{\delta}$, it is clear that the functions -- or measures -- $\mu^{\delta}_{h}$ have precisely the form of the measure $\nu$ in Lemma \ref{comparison_higher} for $d=n-l$. In particular,
\begin{equation}\label{form11_higher} I_{m-l}(\mu_{h}^{\delta}) \lesssim_{t,l,m,n} \delta^{t - (m-l)}I_{t}(\mu_{h}^{\delta}), \qquad 0 < t \leq m-l. \end{equation}
If $V \in G(n-l,m-l)$, write $(\mu_{h}^{\delta})_{V}$ for the orthogonal projection of $\mu_{h}^{\delta}$ onto the subspace $V \subset \R^{n-l}$. Before making the final estimates, we need to record the following upper bound for the energy in terms of the $L^2$-norm.

{
\begin{lemma} Let $\nu$ be a positive finite compactly supported Borel measure on $\R^{m-l}$, which is also an $L^{2}$-function. Then,
\begin{equation}\label{form19_higher} I_{t}(\nu) \lesssim_{l,m,t} \|\nu\|_{2}^{2}, \quad 0 < t < m-l, \end{equation}
\end{lemma}
\begin{proof}
The Fourier transform of the finite positive measure $\nu$ is a positive definite function, so it satisfies the pointwise estimate $|\hat{\nu}(x)| \leq \hat{\nu}(0)$ for $x \in \R^{m - l}$, see for instance \cite[p. 198, p.220]{Bo}. Using this we obtain
\begin{align*}
I_{t}(\nu) & \asymp_{l,m,t} \int_{\R^{m - l}} |\hat{\nu}(x)|^{2}|x|^{t - (m - l)} \, dx\\
& \leq |\hat{\nu}(0)|^{2}\int_{B(0,1)} |x|^{t-(m-l)}dx + \int_{\R^{m - l} \setminus B(0,1)} |\hat{\nu}(x)|^{2} \, dx\\
& \lesssim_{l,m,t} \|\nu\|_{1}^{2} + \|\hat{\nu}\|_{2}^{2} \lesssim \|\nu\|_2^2,
\end{align*}
as claimed. 
\end{proof}}

We will also apply the estimate
\begin{align}\label{form20_higher} \int_{G(n-l,m-l)} \|\nu_{V}\|_{2}^{2} \, d\gamma_{n-l,m-l}(V)  \lesssim  I_{m-l}(\nu), \end{align}
valid for any finite measure $\nu$ on $\R^{n-l}$ with $I_{m-l}(\nu) < \infty$. This can be seen for instance from Theorem 3.1 in \cite{Mat1}.

Let $0 < t < m-l$. Combining \eqref{form20_higher}, \eqref{form11_higher} and \eqref{form19_higher}
 we have
\begin{align*} \int_{G(n-l,m-l)} \|(\mu_{\mathbb{V}})_{\delta}\|_{2}^{2} \, d\gamma_{n-l,m-l}(V) & \lesssim \int_{G(n-l,m-l)} \int_{\R^l} \|(\mu^{\delta}_{h})_{V}\|_{2}^{2} \, dh \, d\gamma_{n-l,m-l}(V)\\
& = \int_{\R^l} \int_{G(n-l,m-l)} \|(\mu_{h}^{\delta})_{V}\|_{2}^{2} \, d\gamma_{n-l,m-l}(V) \, dh\\
& {\lesssim} \int_{\R^l} I_{m-l}(\mu_{h}^{\delta}) \, dh \\& {\lesssim_{t}} \delta^{t - (m-l)}\int_{\R^l} I_{t}(\mu_{h}^{\delta}) \, dh\\
& \leq \delta^{t - (m-l)}\int_{\R^l} I_{t}((\mu_{h}^{\delta})_{V_0}) \, dh\\& {\lesssim_{t}} \delta^{t - (m-l)} \int_{\R^l} \|(\mu_{h}^{\delta})_{V_0}\|_{2}^{2} \, dh\\
& = \delta^{t - (m-l)}\|(\mu^{\delta})_{\mathbb{V}_0}\|_{2}^{2} \lesssim \delta^{t - (m-l) + \sigma - m}. \end{align*}

If $s < \sigma$, we may choose $t < m-l$ so close to $m-l$ that $s < t - (m-l) + \sigma$. By Chebyshev's inequality
\begin{align*}
\gamma_{n-l,m-l} (\{V\in G(n-l,m-l) : & \|(\mu_{\mathbb{V}})_{\delta}\|_{2}^{2} \geq \delta^{s - m}\})\\&\leq \frac{1}{\delta^{s-m}} \int_{G(n-l,m-l)}\|(\mu_{\mathbb{V}})_{\delta}\|_2^2\,d\gamma_{n-l,m-l}(V)
\end{align*}
we obtain
\begin{displaymath} \gamma_{n-l,m-l}(\{V \in G(n-l,m-l) : \|(\mu_{\mathbb{V}})_{\delta}\|_{2}^{2} \geq \delta^{s - m}\}) \lesssim \delta^{t - (m-l) + \sigma - s}. \end{displaymath}

For $\delta_j = 2^{-j}$, $j \in \N$, combining this estimate with the easier Borel-Cantelli lemma shows that
\begin{displaymath}
\gamma_{n-l,m-l}\left(\bigcap_{p=1}^{\infty}\bigcup_{j\geq p}\{V\in G(n-l,m-l) : \|(\mu_{\mathbb{V}})_{\delta_j}\|_2^2\geq \delta_j^{s-m}\}\right)=0
\end{displaymath}
 and thus that the inequality $\|(\mu_{\mathbb{V}})_{2^{-j}}\|_{2}^{2} \geq 2^{j(m - s)}$ can hold \textbf{infinitely often} only for a set of $V$'s of $\gamma_{n-l,m-l}$ measure zero. For the rest of the subspaces $V$, we have $\|(\mu_{\mathbb{V}})_{2^{-j}}\|_{2}^{2} \lesssim_{V} 2^{j(m - s)}$ for $j \in \N$, and, according to Lemma \ref{L1_higher}, this implies $\Hd B_{\mathbb{V}} \geq s$ for every such $V\in G(n-l,m-l)$.
\end{proof}

\section{Proofs for upper box and packing dimensions}

A quick word on notation before we begin. If $E \subset \R^{n}$ is a bounded set and $\delta > 0$, we denote by $N(E,\delta)$ the least number of (closed) balls of radius $\delta$ required to cover $E$. The upper and lower box dimensions (Minkowski dimensions) of $E$ are defined by
\begin{displaymath} \underline{\dim}_{\B} E := \liminf_{\delta \to 0} \frac{\log N(E,\delta)}{-\log \delta} \quad \text{and} \quad \overline{\dim}_{\B} E := \limsup_{\delta \to 0} \frac{\log N(E,\delta)}{-\log \delta}. \end{displaymath}
Analogous definitions can be made for totally bounded sets in metric spaces, for instance, we will use the concept of box dimensions on the Grassmanian.

The packing dimension of a set $E\subset \R^{n}$ is defined as
\begin{displaymath}
\dim_{\p} E := \inf \left\{\sup_j \overline{\dim}_{\B}F_j:\; E \subset \bigcup_{j\in \N}F_j\right\}.
\end{displaymath}

{Theorem \ref{main1_higher} contains statements concerning both upper box and packing dimension; accordingly, our proof divides into two parts. However, it turns out that the assertions for packing dimension easily reduce to their analogues for upper box dimension (via Lemma \ref{OrLemma}), so all the main ingredients of the proof are contained in the first part.}

{Let us briefly explain these ingredients in the lowest-dimensional interesting case, namely when $n = 3, m = 2$ and $l = 1$. Thus, we are considering projections in $\R^{3}$ onto the `vertical' $2$-dimensional subspaces (containing the $z$-axis). The key observation is, in fact, a result concerning planar sets and their projections onto one-dimensional subspaces. Fix $\delta > 0$, and let $K \subset \R^{2}$ be a bounded set. Suppose that for \textbf{some} one-dimensional subspace $L \subset \R^{2}$ the projection of $K$ onto $L$ contains $N \in \N$ $\delta$-separated points. Then, the conclusion is that for `almost' every one-dimensional subspace in $\R^{2}$ the projection of $K$ contains $\gtrsim N$ $\delta$-separated points (where the correct interpretation of $\gtrsim$ slightly differs from our normal usage).}

{How do we use this observation for sets in $\R^{3}$? To begin with, we slice our bounded set $B \subset \R^{3}$ into disjoint horizontal pieces $B_{H}$ of height $\delta$. These pieces are `planar enough'  for the observation above to be applied. Namely, a moment's thought reveals that, at scale $\delta$, the projections of the horizontal pieces onto the vertical subspaces $\mathbb{V} \subset \R^{3}$ resemble projections of certain planar sets onto one-dimensional subspaces. In particular, if $N(\pi_{\mathbb{V}_{0}}(B_{H}),\delta) = N \in \N$ for \textbf{some} vertical subspace $\mathbb{V}_{0} \subset \R^{3}$, then the inequality $N(\pi_{\mathbb{V}}(B_{H}),\delta) \gtrsim N$ holds, in a suitable sense, for `almost' every vertical subspace $\mathbb{V} \subset \R^{3}$. Since $N(\pi_{\mathbb{V}}(B),\delta)$ is roughly the sum of the numbers $N(\pi_{\mathbb{V}}(B_{H}),\delta)$ over all the horizontal pieces $B_{H}$, this property of the sets $B_{H}$ transfers easily to the same property for the entire set $B$: if $N(\pi_{\mathbb{V}_{0}}(B),\delta) = N$ for \textbf{some} vertical subspace $\mathbb{V}_{0} \subset \R^{3}$, then $N(\pi_{\mathbb{V}}(B),\delta) \gtrsim N$ for `almost' every vertical subspace $\mathbb{V} \subset \R^{3}$. The assertion of Theorem \ref{main1_higher} for upper box dimension follows immediately. }


We begin with an estimate for the volumes of balls on the Grassmannian. In all likelihood, the proposition is well-known, but we were unable to find a direct reference. Consequently, we chose to include a proof in Appendix \ref{volumes}.


\begin{proposition}\label{p:ball_est}
Let $0<m<n$. Then there exist constants $0<c<C<\infty$ and $\delta_0>0$ such that
\begin{displaymath}
c \delta^{m(n - m)} \leq \gamma_{n,m}(B(V,\delta)) \leq C \delta^{m(n - m)}
\end{displaymath}
for all $V\in G(n,m)$ and all $0<\delta <\delta_0$. Here the ball $B(V,\delta)$ is defined using the projection distance $d_{\pi}(V,W) = \|\pi_{V} - \pi_{W}\|$.
\end{proposition}


A set $E\subset G(n,m)$ is said to be \emph{$\delta$-separated} if $d_{\pi}(V,W)\geq \delta$ for any distinct elements $V,W \in E$.

\begin{definition}[$(\delta,k)$-sets]\label{d:d_k_set} Let $C \subset B(0,1) \subset \R^{n}$ be a finite set. We say that a $\delta$-separated set $C$ is a \emph{$(\delta,k)$-set}, if
\begin{displaymath} \card [B(x,r) \cap C] \lesssim \left(\frac{r}{\delta}\right)^k \end{displaymath}
for every ball $B(x,r) \subset \R^{n}$ with radius $r \geq \delta$.
\end{definition}

The following proposition is a generalization of \cite[Proposition 4.10]{Or} to higher dimensions. Essentially, the result is a discrete version of the Marstrand-Kaufman-Mattila projection theorem.

\begin{lemma}\label{OR} Let $0<\delta<1$ and let $C \subset B(0,1) \subset \R^{n}$ be a $(\delta,m)$-set with $N \in \N$ points. Let $\tau > 0$, and let $E \subset G(n,m)$ be a $\delta$-separated collection of subspaces such that
\begin{displaymath} N(C_{V},\delta) \leq \delta^{\tau}N, \qquad \text{for all } V \in E. \end{displaymath}
Then $\card E \lesssim \delta^{\tau-(n-m)m} \cdot \log(1/\delta)$.
\end{lemma}

\begin{proof} Let $\calD_{\delta}$ be a partition of $V$ into $m$-dimensional dyadic cubes.
For a given subspace $V\in E$ we consider the `tube'
\begin{displaymath}
\calT_V:= \{T=\pi_V^{-1}(Q):\; Q\in \calD_{\delta}\},
\end{displaymath}
and we define the relation
\begin{displaymath}
x \sim_V y \quad \Leftrightarrow \quad x,y\in T \in \calT_V.
\end{displaymath}
We define an energy $\calE$ by
\begin{displaymath}
\calE:= \sum_{V \in E}\mathrm{card}\{(x,y)\in C \times C:\; x \sim_V y\}.
\end{displaymath}
Writing
\begin{displaymath}
\calE':= \sum_{V \in E}\mathrm{card}\{(x,y)\in C \times C:\; x \sim_V y,\,x\neq y\},
\end{displaymath}
we find that $\calE = \calE' + N \cdot \mathrm{card}E$, and our goal is to show that $\calE \lesssim \delta^{-(n-m)m}\cdot N \cdot \log \left(\frac{1}{\delta}\right)$. Proposition \ref{p:ball_est} implies that $\mathrm{card}E \lesssim \delta^{-(n-m)m}$ since $E$ is a $\delta$-separated subset of $G(n,m)$. So, it remains to establish the desired upper bound for $\calE'$.

Let us first observe that
\begin{equation}\label{eq:est2}
\mathrm{card}\{V\in E:\; x\sim_V y,\;x\neq y\}\lesssim \frac{\delta^{m(1-(n-m))}}{|x-y|^{m}}.
\end{equation}
Namely, if $V$ is any subspace such that $x \sim_{V} y$, then
\begin{displaymath} B(V,\delta) \subset \{V : |\pi_{V}(x - y)| \leq \beta\delta\} \end{displaymath}
for some constant $\beta$ depending only on $m$ and $n$ (here we also use the inclusion $C \subset B(0,1)$). On the other hand, we have the measure bound
\begin{align*}
\gamma_{n,m}(\{V \in G(n,m) \colon |\pi_V(x-y)|\leq \beta\delta\}) \lesssim \left(\tfrac{\delta}{|x-y|}\right)^{m}.
\end{align*}
Now \eqref{eq:est2} follows, since the set $E$ is $\delta$-separated and, according to Proposition \ref{p:ball_est},  we have $\gamma_{n,m}(B(V,\delta)) \asymp \delta^{m(n - m)}$.

Using \eqref{eq:est2},
\begin{align*}
\calE'&=\sum_{x\in C} \sum_{j:\delta\leq 2^j \leq 1}\sum_{y\in C\atop 2^j \leq |x-y|<2^{j+1}}\mathrm{card}\{V\in E:\;x\sim_V y\}\\
&\lesssim \sum_{x\in C}\sum_{j:\delta\leq 2^j\leq 1}\sum_{y\in C\atop 2^j\leq |x-y|<2^{j+1}}|x-y|^{-m}\delta^{m(1-(n-m))}\\
&\lesssim \sum_{x\in C} \sum_{j:\delta \leq 2^j \leq 1}\mathrm{card}[C\cap B(x,2^{j+1})]\cdot 2^{-jm}\delta^{m(1-(n-m))}\\
&\lesssim \sum_{x\in C}\sum_{j:\delta\leq 2^j \leq 1}\left(\frac{2^{j+1}}{\delta}\right)^{m}2^{-jm}\delta^{m(1-(n-m))}\\
&= \sum_{x\in C}\sum_{j:\delta \leq 2^j \leq 1}\delta^{-(n-m)m} \asymp \delta^{-(n-m)m}\cdot N\cdot \log \left(\frac{1}{\delta}\right).
\end{align*}

The asserted bound  for $\mathrm{card} E$ follows, once we have found an appropriate lower bound for $\calE$. We may assume that $\delta^{\tau}N\geq 1$. The assumption $N(C_V,\delta)\leq \delta^{\tau}N$ guarantees that $C$ can be covered by $K\lesssim \delta^{\tau}N$ tubes $T_1,\ldots,T_K \in \calT_V$, which yields
\begin{align*}
\mathrm{card}\{(x,y)\in C\times C:\; x \sim_V y\}&= \sum_{j=1}^K \mathrm{card}\{(x,y)\in C \times C:\; x,y \in T_j\}\\
&= \sum_{j=1}^K \mathrm{card}[C \cap T_j]^2\\
&\geq \frac{1}{K} \left(\sum_{j=1}^K \mathrm{card}[C\cap T_j]\right)^2\\
& \gtrsim \delta^{-\tau} \cdot N^{-1} \cdot (\mathrm{card}C)^2 = \delta^{-\tau}\cdot N.
\end{align*}
Combing the upper and lower bounds for $\calE$, we find
\begin{displaymath}
\delta^{-\tau}\cdot N \cdot \mathrm{card}E \lesssim \calE \lesssim \delta^{-(n-m)m}\cdot N \cdot \log \left(\tfrac{1}{\delta}\right),
\end{displaymath}
which is the desired result.
\end{proof}

The following reformulation of the lemma will be used later (applied to the Grassmanian $G(n-l,m-l)$ instead of $G(n,m)$).

\begin{cor}\label{ORCor} Let $C \subset B(0,1) \subset \R^{n}$ be a $(\delta,m)$-set with $N \in \N$ points. Then, if $E \subset G(n,m)$ is any $\delta$-separated set with $\card E \geq \delta^{-\beta}$ elements, we have
\begin{displaymath} \frac{1}{\card E} \sum_{V \in E} N(C_{V},\delta) \gtrsim_{\tau} \delta^{(n-m)m - \tau}N, \qquad \tau < \beta. \end{displaymath}
\end{cor}
\begin{proof} According to Lemma \ref{OR}, the set $E$ contains at most
\begin{displaymath}
\lesssim \delta^{((n-m)m - \tau) - (n-m)m} \cdot \log (1/\delta) = \delta^{-\tau} \cdot \log (1/\delta)
\end{displaymath}
 subspaces $V$ such that $N(C_{V},\delta) \leq \delta^{(n-m)m - \tau}N$. Since $\tau < \beta$, the proportion of such subspaces in $E$ is close to zero for small $\delta$, and the claim follows. \end{proof}

\subsection{Proof of Theorem \ref{main1_higher} for upper box dimension}

We are now ready to prove Theorem \ref{main1_higher} for upper box dimension. The assumption on the analyticity of the set $B$ will only be required later, in the proof for packing dimension. For the time being, we assume that $B \subset B(0,1) \subset \R^{n}$ is an arbitrary set with
\begin{displaymath} \m_{\B} = \sup\{\overline{\dim}_{\B} B_{\mathbb{V}} : V \in G(n-l,m-l)\} > 0. \end{displaymath}
We will show for $0 \leq \sigma \leq \m_{\B}$ that
\begin{equation}\label{form12_higher} \underline{\dim}_{\MB} \{V \in G(n-l,m-l) : \overline{\dim}_{\B} B_{\mathbb{V}} < \sigma\} \leq \max\{0,(n-m)(m-l) + \sigma - \m_{\B}\}, \end{equation}
where $\underline{\dim}_{\MB}$ denotes the \emph{modified lower box dimension}
\begin{displaymath} \underline{\dim}_{\MB} E := \inf\left\{\sup_{j} \underline{\dim}_{\B} F_{j} : E \subset \bigcup_{j \in \N} F_{j} \right\}. \end{displaymath}
Recall that $G(n-l,m-l)$ is endowed with a metric so that $\Hd G(n-l,m-l)=(n-m)(m-l)$ and the $(n-m)(m-l)$-dimensional Hausdorff measure coincides with $\gamma_{n-l,m-l}$ up to a positive and finite multiplicative constant.
It is clear that sets $E \subset G(n-l,m-l)$ with
\begin{displaymath}
\underline{\dim}_{\MB} E < (n-m)(m-l)
\end{displaymath}
are meager, i.e., countable unions of nowhere dense sets, and have $\gamma_{n-l,m-l}$ measure zero, so \eqref{form12_higher} will imply the upper box dimension part of Theorem \ref{main1_higher}.

As before, let $\calD_{\delta}$ stand for the collection of dyadic cubes in $\R^{n}$ of side-length $\delta > 0$. Write
\begin{displaymath} B^{\delta} := \bigcup\{Q \in \calD_{\delta} : B \cap Q \neq \emptyset\}, \quad \delta > 0, \end{displaymath}
It is easy to check that
\begin{displaymath} N((B^{\delta})_{\mathbb{V}},\delta) \asymp N(B_{\mathbb{V}},\delta) \end{displaymath}
for any $V \in G(n-l,m-l)$ and $\delta > 0$. This shows that $\limsup_{\delta \to 0} \frac{\log N(B_{\mathbb{V}},\delta)}{-\log \delta}= \limsup_{\delta\to 0} \frac{\log N((B^{\delta})_{\mathbb{V}},\delta)}{-\log \delta}$ and thus
\begin{align*} \{V \in G(n-l,m-l) &: \overline{\dim}_{\B} B_{\mathbb{V}} < \sigma\} \\&\subset \bigcup_{i \in \N} \bigcap_{\delta \in (0,1/i)} \{V \in G(n-l,m-l) : N((B^{\delta})_{\mathbb{V}},\delta) \leq \delta^{-\sigma}\}. \end{align*}
Hence, by definition of $\underline{\dim}_{\MB}$, the bound \eqref{form12_higher} would follow from
\begin{equation}\label{form13_higher} \sup_{i} \underline{\dim}_{\B} E_{i} \leq \max\{0,(n-m)(m-l) + \sigma - \m_{\B}\}, \qquad 0 \leq \sigma \leq \m_{\B}, \end{equation}
where $E_{i} = \bigcap_{\delta \in (0,1/i)} \{V \in G(n-l,m-l) : N((B^{\delta})_{\mathbb{V}},\delta) \leq \delta^{-\sigma}\}$. We will now prove \eqref{form13_higher}. Fix $i \in \N$ and write $E := E_{i}$. Given $\sigma < \sigma' < \m_{\B}$, we may find a direction $V_0 \in G(n-l,m-l)$ and a sequence $(\delta_{j})_{j \in \N}$ such that $\delta_{j} \searrow 0$, and $N((B^{\delta_{j}})_{\mathbb{V}_0},\delta_{j}) \geq \delta_{j}^{-\sigma'}$.

\subsubsection{Decomposition into sets essentially in $\R^{n - l}$} Let
\begin{align*}
\calH_{\delta}&:=\{H=\R^{n-l}\times \Pi_{i=1}^l [k_i \delta, (k_i +1)\delta): (k_{1},\ldots,k_{l}) \in \Z^{l}\}
\end{align*}
and set $B^{\delta,H}:= B^{\delta}\cap H$. Thus
\begin{displaymath}
B^{\delta}=\bigcup_{H\in\mathcal{H}_{\delta}} B^{\delta,H}.
\end{displaymath}
In particular,
\begin{equation}\label{form14_higher} N((B^{\delta})_{\mathbb{V}},\delta) \asymp \sum_{H \in \calH_{\delta}} N((B^{\delta,H})_{\mathbb{V}},\delta) \end{equation}
for $V \in G(n-l,m-l)$ and $\delta > 0$. For our purposes, the sets $B^{\delta,H}$ are essentially sets in $\R^{n-l}$ in the following sense: for each $H \in \calH_{\delta}$, there exists a set $P^{\delta,H}\subset \R^{n-l}$ and an $l$-tuple  $(k_1,\ldots,k_l) \in \Z^{l}$ such that $B^{\delta,H} = P^{\delta,H} \times\Pi_{i=1}^l [k_i\delta,(k_i+1)\delta)$. Moreover, the projection properties of the sets $B^{\delta,H}$ and $P^{\delta,H}$ are equivalent in the sense that $N((P^{\delta,H})_{{V}},\delta) \asymp N((B^{\delta,H})_{\mathbb{V}},\delta)$ for $V \in G(n-l,m-l)$ and $\delta > 0$, where $(P^{\delta,H})_{{V}}$ is the orthogonal projection of $P^{\delta,H}$ onto the $(m-l)$-dimensional subspace $V\subset \R^{n-l}$.

\subsubsection{Finding $(\delta,m-l)$-sets}  In order to apply Corollary \ref{ORCor} to $G(n-l,m-l)$, we need to extract some $(\delta,m-l)$-sets. Recall the special direction $V_0 \in G(n-l,m-l)$ with the property that $N((B^{\delta})_{\mathbb{V}_0},\delta_j) \geq \delta_j^{-\sigma'}$ for every $j\in\mathbb{N}$. Fix $j \in \N$ and write $\delta := \delta_{j}$. For $H \in \calH_{\delta}$, we set
\begin{displaymath} M_{H} := N((B^{\delta,H})_{\mathbb{V}_0},\delta) \asymp N((P^{\delta,H})_{V_0},\delta). \end{displaymath}
Then $\sum_{H \in \calH^{\delta}} M_{H} \gtrsim \delta^{-\sigma'}$ according to \eqref{form14_higher}. As before, let $\calT_{V_0}$ be a partition of $\R^{n-l}$ into tubes perpendicular to $V_0$. To be precise, let $\calD_{\delta}$ be a partition of $V_0$ into dyadic cubes and then consider
\begin{displaymath}
\mathcal{T}_{V_0}:=\{T=\pi_{V_0}^{-1}(Q):\; Q\in \calD_{\delta}\}.
\end{displaymath}

Since $N((P^{\delta,H})_{V_0},\delta) \asymp M_{H}$, we may find $K \asymp M_{H}$ tubes $T_{1},\ldots,T_{K} \in \calT_{V_0}$ such that each tube $T_{k}$, $k=1,\ldots,K$, contains a point $x_{k} \in P^{\delta,H}$ and such that the set $C^{\delta,H} := \{x_{1},\ldots,x_{K}\}$ is $\delta$-separated. Moreover, since any ball $B(x,r) \subset \R^{n-l}$ of radius $r \geq \delta$ intersects no more than $\lesssim (r/\delta)^{m-l}$ tubes in $\calT_{V_0}$, we may infer that the set $C^{\delta,H}$ is a $(\delta,m-l)$-set containing $\card C^{\delta,H} \asymp M_{H}$ elements.

\subsubsection{Concluding the proof for upper box dimension} Write $\delta = \delta_{j}$, where $\delta_j$ is as before. We apply Corollary \ref{ORCor} to the sets $C^{\delta,H}$, for every $H \in \calH_{\delta}$. Let $E \subset G(n-l,m-l)$ be any $\delta$-separated set of cardinality $\card E \geq \delta^{-\beta}$, for some $\beta > 0$. Then,
\begin{align*} \frac{1}{\card E} \sum_{V \in E} N((B^{\delta})_{\mathbb{V}},\delta) & \gtrsim \sum_{H \in \calH_{\delta}} \left( \frac{1}{\card E} \sum_{V \in E} N((C^{\delta,H})_{V},\delta) \right)\\
& \gtrsim \delta^{(n-m)(m-l) - \tau} \cdot \sum_{H \in \calH_{\delta}} M_{H} \gtrsim \delta^{(n-m)(m-l) - \tau - \sigma'}, \quad \tau < \beta. \end{align*}

What does this mean? Recall that $\sigma < \sigma' < \m_{\B}$. If
\begin{displaymath}
\beta > \max\{0,(n-m)(m-l) + \sigma - \sigma'\},
\end{displaymath}
 we may apply the previous estimate with some
\begin{displaymath}
\tau > \max\{0,(n-m)(m-l)+ \sigma - \sigma'\}
\end{displaymath}
 to obtain the inequality
\begin{displaymath} \frac{1}{\card E} \sum_{V \in E} N((B^{\delta})_{\mathbb{V}},\delta) > \delta^{-\sigma}, \end{displaymath}
at least for $\delta = \delta_{j}$ small enough.

This implies that
\begin{displaymath} E \not\subset \{V : N((B^{\delta})_{\mathbb{V}},\delta) \leq \delta^{-\sigma}\}. \end{displaymath}
Thus, for small enough $\delta = \delta_{j}$, the maximum cardinality of a $\delta$-separated subset of $\{V : N((B^{\delta})_{\mathbb{V}},\delta) \leq \delta^{-\sigma}\}$ is less than $\delta^{-\beta}$, for any
\begin{displaymath}
\beta > \max\{0,(n-m)(m-l)+ \sigma - \sigma'\}.
\end{displaymath}
 Since $\sigma'< \m_{\B}$ was arbitrary, this yields \eqref{form13_higher} and completes the proof of Theorem \ref{main1_higher} for upper box dimension.

\subsection{Proof of Theorem \ref{main1_higher} for packing dimension}

Let $B \subset \R^{n}$ be a bounded analytic set, and assume that
\begin{displaymath} \m_{\p} := \sup\{\dim_{\p} B_{\mathbb{V}} : V \in G(n-l,m-l)\} > 0. \end{displaymath}
As in the case of upper box dimension, it suffices to prove for $0 \leq \sigma \leq \m_{\p}$ that
\begin{equation}\label{form15_higher} \underline{\dim}_{\MB} \{V \in G(n-l,m-l) : \dim_{\p} B_{\mathbb{V}} < \sigma\} \leq \max\{0,(n-m)(m-l) + \sigma - \m_{\p}\}. \end{equation}
Suppose that \eqref{form15_higher} fails. Then, we may find numbers $0 < \sigma < \sigma' < \m_{\p}$ such that
\begin{equation}\label{form17} \underline{\dim}_{\MB} \{V \in G(n-l,m-l) : \dim_{\p} B_{\mathbb{V}} < \sigma\} > \max\{0,(n-m)(m-l) + \sigma - \sigma'\}. \end{equation}
Choose $V_0 \in G(n-l,m-l)$ such that $\dim_{\p} B_{\mathbb{V}_0} > \sigma'$. Since $B_{\mathbb{V}_0} \subset \R^{m}$ is analytic, a result of Joyce and Preiss \cite{JP} permits us to find a compact set $K_{\mathbb{V}_0} \subset B_{\mathbb{V}_0}$ with positive and finite $\sigma'$-dimensional packing measure; $0 < \mathcal{P}^{\sigma'}(K_{\mathbb{V}_0}) < \infty$. Next, we apply the `pull-back lemma' by Lubin to find a Borel measure $\mu$ supported on $B$, with the property that
\begin{equation}\label{form16} \pi_{{\mathbb{V}_0}\sharp}\mu = \mathcal{P}^{\sigma'}\llcorner K_{\mathbb{V}_0}. \end{equation}
Now $B^{0} := \spt \mu \subset B$ is a $\mu$-measurable set with $\mu(B^{0}) > 0$, and \eqref{form17} holds, by monotonicity, with $B$ replaced by $B^{0}$. We quote a lemma from \cite{Or}.
\begin{lemma}[Adapted from Lemma 4.5 in \cite{Or}]\label{OrLemma} Let $\mu$ be a Borel regular measure on $\R^{n}$, and let $\beta,\sigma > 0$. Assume that $B^0 \subset \R^{n}$ is $\mu$-measurable with $0 < \mu(B^0) < \infty$, and
\begin{displaymath} \underline{\dim}_{\MB} \{V \in G(n-l,m-l) : \dim_{\p} B^0_{\mathbb{V}} < \sigma\} > \beta. \end{displaymath}
Then, there exists a $\mu$-measurable set $B' \subset B^0$ with $\mu(B') > 0$ such that
\begin{displaymath} \underline{\dim}_{\MB} \{V\in G(n-l,m-l) : \overline{\dim}_{\B} B'_{\mathbb{V}} < \sigma\} > \beta. \end{displaymath}
\end{lemma}
The corresponding lemma in \cite{Or} only concerns projections of planar sets, but the proof works verbatim in the situation above. We intend to apply the lemma to the measure $\mu$ constructed above and the set $B^{0} = \spt \mu$. Strictly speaking, the abstract `pull-back lemma' from \cite{Lu} does not tell us that the measure $\mu$ is Borel \textbf{regular}. However, inspecting the proof of Lemma 4.5 in \cite{Or}, the regularity of the measure is only used to guarantee the existence of compact sets $K \subset B^0 \cap \spt \mu$ with positive $\mu$-measure. Fortunately, the existence of such sets is clear in our situation, since here  $ B^{0}=\spt \mu$ is closed to begin with.

Applying Lemma \ref{OrLemma} to the measure $\mu$ constructed above, we find a set $B'$ in $B^{0}$ such that $\mu(B') > 0$, and
\begin{equation}\label{form18_higher} \underline{\dim}_{\MB} \{V \in G(n-l,m-l) : \overline{\dim}_{\B} B'_{\mathbb{V}} < \sigma\} > \max\{0,(n-m)(m-l) + \sigma - \sigma'\}. \end{equation}
However, we may infer from \eqref{form16} that
\begin{displaymath} \mathcal{P}^{\sigma'}(B'_{\mathbb{V}_0}) = \mu(\pi_{\mathbb{V}_0}^{-1}(B'_{\mathbb{V}_0})) \geq \mu(B') > 0, \end{displaymath}
and, in particular,
\begin{displaymath} \m_{\B}' := \sup\{\overline{\dim}_{\B} B'_{\mathbb{V}} : V \in G(n-l,m-l)\} \geq \overline{\dim}_{\B} B'_{\mathbb{V}_0} \geq \sigma'. \end{displaymath}
Now it follows from the upper box dimension part of the proof, namely the estimate \eqref{form12_higher}, that
\begin{align*} \underline{\dim}_{\MB} \{V \in G(n-l,m-l) : \overline{\dim}_{\B} B'_{\mathbb{V}} < \sigma\} &\leq \max\{0,(n-m)(m-l) + \sigma - \m_{\B}'\}\\& \leq \max\{0,(n-m)(m-l) + \sigma - \sigma'\}.
\end{align*}
This contradicts \eqref{form18_higher} and concludes the proof of Theorem \ref{main1_higher} for bounded sets, and, according to Remark \ref{bounded}, for all sets.

\appendix

\section{Volumes of balls on the Grassmannian}\label{volumes}

In this final section, we prove Proposition \ref{p:ball_est}. We start with a geometric lemma.

\begin{lemma}\label{bases} Let $V,W \in G(n,m)$. Then there exist orthonormal bases $\{v_{1},\ldots,v_{m}\} \subset V$ and $\{w_{1},\ldots,w_{m}\} \subset W$ such that
\begin{displaymath} |v_{i} - w_{i}| \lesssim \|\pi_{V} - \pi_{W}\|, \qquad 1 \leq i \leq m. \end{displaymath}
\end{lemma}

\begin{proof} Write $\epsilon := \|\pi_{V} - \pi_{W}\|$. Choose some orthonormal bases for $V$ and $W$, and form the $(n \times m)$-matrices $Q_{V}$ and $Q_{W}$ with the basis vectors as columns. Then $Q_{V}^{\mathrm{T}} = \pi_{V}$, and $Q_{V}$ maps $\R^{m}$ isometrically onto $V$, as follows from
\begin{displaymath} |x| = |\operatorname{Id}x| = |Q_{V}^{\mathrm{T}}Q_{V}x| \leq |Q_{V}x| \leq |x|, \qquad x \in \R^{m}. \end{displaymath}
Similar statements hold for $Q_{W}$. Consider the $(m \times m)$-matrix $M := Q_{V}^{\mathrm{T}}Q_{W}$. If $\epsilon < 1$, as we may assume, $M$ is nonsingular; otherwise one finds a unit vector $x \in \operatorname{ker} M$, and then $|\pi_{V}(Q_{W})x - \pi_{W}(Q_{W})x| = 1 > \epsilon$.  We perform the singular value decomposition (SVD) for $M$:
\begin{displaymath} M = O_{1}\Sigma O_{2}^{\mathrm{T}}. \end{displaymath}
Here $O_{1},O_{2} \in O(m)$, since $\det M \neq 0$, and $\Sigma$ is a diagonal $(m \times m)$-matrix with non-negative entries, namely the \emph{singular values} of $M$. We first aim to bound the singular values from below. Let $x \in \R^{m}$ be an arbitrary unit vector. Since $\|\pi_{V} - \pi_{W}\| = \epsilon$, we have
\begin{displaymath} |Mx| = |\pi_{V}(Q_{W}x)| \geq |\pi_{W}(Q_{W}x)| - |\pi_{V}(Q_{W}x) - \pi_{W}(Q_{W}x)| \geq 1 - \epsilon, \end{displaymath}
using the fact that $Q_{W}x$ is a unit vector on $W$. Now, fix $1 \leq j \leq m$ and choose the unit vector $x \in \R^{m}$ so that $O_{2}^{\mathrm{T}}x$ equals the $j^{{th}}$ standard basis vector $e_{j}$. Then
\begin{displaymath} 1 - \epsilon \leq |Mx| = |O_{1}\Sigma O_{2}^{\mathrm{T}}x| = |\sigma_{j}O_{1}e_{j}| = \sigma_{j}, \end{displaymath}
where $\sigma_{j}$ is the $j^{th}$ diagonal element in $\Sigma$ -- the $j^{th}$ singular value. In conclusion, all the singular values $\sigma_{j}$ satisfy $\sigma_{j} \geq 1 - \epsilon$. Now we are prepared to construct the bases. The SVD implies that
\begin{displaymath} [Q_{V}O_{1}]^{\mathrm{T}}[Q_{W}O_{2}] = \Sigma. \end{displaymath}
We simply observe that the columns of the $(n \times m)$-matrices $Q_{V}O_{1}$ and $Q_{W}O_{2}$ form orthonormal bases $\{v_{1},\ldots,v_{m}\}$ and $\{w_{1},\ldots,w_{m}\}$ for the subspaces $V$ and $W$, respectively. Moreover, the inner product of any pair $(v_{i},w_{j})$ satisfies
\begin{displaymath} v_{i} \cdot w_{j} = \sigma_{j}\delta_{ij} \geq (1 - \epsilon)\delta_{ij}. \end{displaymath}
This means that the angles between the vectors $v_{i}$ and $w_{i}$, $1 \leq i \leq m$, are $\lesssim \epsilon$, and the rest follows by simple trigonometry.
\end{proof}

The measure $\gamma_{n,m}$ is $O(n)$-invariant, as follows immediately from the construction, see \cite[\S3.9]{Mat2}. By $O(n)$-invariance, we of course mean that
\begin{displaymath} \gamma_{n,m}(B(V,\delta)) = \gamma_{n,m}(B(OV,\delta)) \end{displaymath}
for any $m$-plane $V \in G(n,m)$, any transformation $O \in O(n)$, and any $\delta > 0$. Since for any pair of $m$-planes $V,W \in G(n,m)$ we may find $O \in O(n)$ with $OV = W$, this allows us to make the following reduction: \textbf{in order to prove Proposition \ref{p:ball_est}, it suffices to find an $m$-plane $V \in G(n,m)$ with}
\begin{equation}\label{form1} 0 < \liminf_{\delta \to 0} \frac{\gamma_{n,m}(B(V,\delta))}{\delta^{m(n - m)}} \leq \limsup_{\delta \to 0} \frac{\gamma_{n,m}(B(V,\delta))}{\delta^{m(n - m)}} < \infty. \end{equation}

\begin{proof}[Proof of \eqref{form1}] Unfortunately, we are not able to prove \eqref{form1} for the measure $\gamma_{n,m}$ directly. Instead, the strategy will be roughly to (i) interpret the Grassmannian $G(n,m)$ as an $m(n - m)$-dimensional smooth submanifold of some Euclidean space, (ii) conclude that the natural Hausdorff measure on the submanifold satisfies a condition analogous to \eqref{form1}, and finally (iii) show that the measure $\gamma_{n,m}$ is equivalent to the said Hausdorff measure.

Step (i) involves the space $\bigwedge_{m} \R^{n}$ of all $m$-vectors over $\R^{n}$. For an introduction to the space $\bigwedge_{m} \R^{n}$, see \cite[Part I, Chapter I]{Wh}. We will mainly need to know that $\bigwedge_{m} \R^{n}$ is an $\binom{n}{m}$-dimensional vector space and can be endowed with a natural inner product, see \cite[\S1.1.12]{Wh}; we denote by $\| \cdot \|_{m}$ the norm induced by this inner product. The vectors in $\bigwedge_{m} \R^{n}$ can be expressed as linear combinations of \emph{simple $m$-vectors} of the form $v_{1} \wedge \cdots \wedge v_{m}$ where $v_{1},\ldots,v_{m} \in \R^{m}$ and $\wedge$ is the wedge product. The subset
\begin{displaymath} G := \{\mathbf{w} : \mathbf{w} \text{ is a simple $m$-vector, and } \|\mathbf{w}\|_{m} = 1\} \end{displaymath}
is a compact smooth $m(n - m)$-dimensional submanifold of $\bigwedge_{m} \R^{n}$, as shown in \cite[\S3.2.28]{Fe}. In particular, if we consider the $m(n - m)$-dimensional Hausdorff measure $\calH^{m(n - m)}$ living on $G \subset \bigwedge_{m} \R^{n}$ -- defined using the norm $\| \cdot \|_{m}$ -- we may conclude that there exists a simple $m$-vector $\mathbf{w}_{0} \in G$ such that
\begin{equation}\label{form2} \lim_{\delta \to 0} \frac{\calH^{m(n - m)}(B(\mathbf{w}_{0},\delta))}{\delta^{m(n - m)}} = \kappa > 0. \end{equation}

Steps (i) and (ii) are now behind us; it only remains to relate $G$ to $G(n,m)$. Consider a pair of vectors $\{-\mathbf{v},\mathbf{v}\} \subset G$. Since $\mathbf{v}$ is simple and $\mathbf{v} \neq 0$, we know that $\mathbf{v} = v_{1} \wedge \cdots \wedge v_{m}$ for some linearly independent vectors $v_{1},\ldots,v_{m} \in \R^{n}$. Hence, the set $\{v_{1},\ldots,v_{m}\}$ spans a subspace $V \in G(n,m)$. We now consider the mapping $T \colon G \to G(n,m)$, defined by $T(\{-\mathbf{v},\mathbf{v}\}) = V$. Our first claims are that \textbf{$T$ is $2$-to-$1$ and surjective}. Let $V \in G(n,m)$, and consider the subspace $L_{V}$ of $\bigwedge_{m} \R^{n}$ spanned by the simple $m$-vectors $v_{1} \wedge \cdots \wedge v_{m}$ with $v_{j} \in V$ for $1 \leq j \leq m$. Since $L_{V} \cong \wedge_{m} \R^{m}$, we infer that $\Hd L_{V} = \binom{m}{m} = 1$. So, $L_{V}$ is a one-dimensional subspace of $\bigwedge_{m} \R^{n}$, and, in particular, $G \cap L_{V} = \{-\mathbf{v},\mathbf{v}\}$ for some vector $\mathbf{v} \in G$. In other words,
\begin{displaymath} T^{-1}(V) = \{-\mathbf{v},\mathbf{v}\}, \end{displaymath}
just as we wanted. This observation allows us to push forward the metric from $G$ to $G(n,m)$ by setting
\begin{equation}\label{form3} d(V,W) := \dist(T^{-1}(V),T^{-1}(W)) = \min\{\|\mathbf{v} - \mathbf{w}\|_{m}, \|\mathbf{v} + \mathbf{w}\|_{m}\}, \end{equation}
provided that $T\mathbf{v} = V$ and $T\mathbf{w} = W$. Of course, $\dist$ refers to the distance with respect to $\| \cdot \|_{m}$. Verifying the triangle inequality for $d$ is an easy case chase using the right hand side of \eqref{form3}. The upshot is that we may now use $d$ to define an $m(n - m)$-dimensional Hausdorff measure $\calH^{m(n - m)}_{d}$ on $G(n,m)$. We now relate $\calH^{m(n - m)}$-densities on $G$ to $\calH^{m(n - m)}_{d}$-densities on $G(n,m)$. In fact, we have
\begin{displaymath} \limsup_{\delta \to 0} \frac{\calH^{m(n - m)}(B(\mathbf{v},\delta))}{\delta^{m(n - m)}} = \limsup_{\delta \to 0} \frac{\calH_{d}^{m(n - m)}(B_{d}(T\mathbf{v},\delta))}{\delta^{m(n - m)}}, \end{displaymath}
and the same equation holds with $\limsup$ replaced by $\liminf$. The proof is simple: given $\mathbf{v}_{0} \in G$, we can find a $\|\cdot\|_{m}$-neighbourhood $U$ of $\mathbf{v}_{0}$ in $G$ so small that
\begin{displaymath} \dist(\{-\mathbf{v},\mathbf{v}\}, \{-\mathbf{w},\mathbf{w}\}) = \|\mathbf{v} - \mathbf{w}\|_{m} \end{displaymath}
for all vectors $\mathbf{v},\mathbf{w} \in U$. Then, the restriction $T|U \colon U \to (T(U),d)$ is an isometry, and
\begin{displaymath} \calH^{m(n - m)}(B(\mathbf{v}_{0},\delta)) = \calH^{m(n - m)}_{d}(T[B(\mathbf{v}_{0},\delta)]) = \calH^{m(n - m)}_{d}(B_{d}(T\mathbf{v}_{0},\delta)) \end{displaymath}
for small enough $\delta > 0$. Recalling \eqref{form2}, we have now proven that
\begin{equation}\label{form4} \lim_{\delta \to 0} \frac{\calH^{m(n - m)}_{d}(B_{d}(W_{0},\delta))}{\delta^{m(n - m)}} = \kappa > 0 \end{equation}
with $W_{0} = T\mathbf{w}_{0}$.

We next need to relate $\calH^{m(n - m)}_{d}$ to $\gamma_{n,m}$. We first consider another Hausdorff measure on $G(n,m)$, namely $\calH^{m(n - m)}_{\pi}$. The letter $\pi$ refers to the projection metric $d_{\pi}(V,W) = \|\pi_{V} - \pi_{W}\|$ on $G(n,m)$. Our aim is to prove that the measures $\calH_{d}^{m(n - m)}$ and $\calH^{m(n - m)}_{\pi}$ are equivalent. To this end, it suffices to demonstrate the bilipschitz-equivalence of the metrics $d$ and $d_{\pi}$:
\begin{equation}\label{form5} cd(V,W) \leq d_{\pi}(V,W) \leq Cd(V,W), \qquad V,W \in G(n,m), \end{equation}
for some positive and finite constants $c$ and $C$. To prove the rightmost inequality, we use the second estimate in \cite[\S1.1.15(7)]{Wh}, namely that if $V,W \in G(n,m)$, and $\mathbf{v},\mathbf{w} \in G$ are $m$-vectors with $T\mathbf{v} = V$ and $T\mathbf{w} = W$, then
\begin{displaymath} |v - \pi_{W}v| \leq \|\mathbf{v} - \mathbf{w}\|_{m} \end{displaymath}
for all unit vectors $v \in V$. Since also $T(-\mathbf{w}) = W$, it follows that
\begin{displaymath} d_{\pi}(V,W) \asymp \sup_{|v| = 1} |v - \pi_{W}v| \leq \min\{\|\mathbf{v} - \mathbf{w}\|_{m},\|\mathbf{v} + \mathbf{w}\|_{m}\} = d(V,W). \end{displaymath}

To prove the leftmost inequality in \eqref{form5}, we fix $V,W \in G(n,m)$ and use Lemma \ref{bases} to find such orthonormal bases $\{v_{1},\ldots,v_{m}\}$ and $\{w_{1},\ldots,w_{m}\}$ for $V$ and $W$ such that $|v_{i} - w_{i}| \lesssim d_{\pi}(V,W)$ for $1 \leq i \leq m$. Then, we use inequality \cite[\S1.12.17]{Wh} to conclude that
\begin{displaymath} d(V,W) \leq \|v_{1} \wedge \cdots \wedge v_{m} - w_{1} \wedge \cdots \wedge w_{m}\|_{m} \lesssim md_{\pi}(V,W). \end{displaymath}
This completes the proof of \eqref{form5}, and shows that $\calH^{m(n - m)}_{d}(B) \asymp \calH^{m(n - m)}_{\pi}(B)$ for any ball $B \subset G(n,m)$ (in either metric). From \eqref{form4}, we may now infer that
\begin{equation}\label{form30} 0 < \liminf_{\delta \to 0} \frac{\calH^{m(n - m)}_{\pi}(B_{\pi}(W_{0},\delta))}{\delta^{m(n - m)}} \leq \limsup_{\delta \to 0} \frac{\calH^{m(n - m)}_{\pi}(B_{\pi}(W_{0},\delta))}{\delta^{m(n - m)}} < \infty. \end{equation}
Finally, we observe that $\calH^{m(n - m)}_{\pi}$ is a finite $O(n)$-invariant measure on $G(n,m)$. The finiteness part follows from the equivalence of $\calH^{m(n - m)}_{\pi}$ with $\calH^{m(n - m)}_{d}$, combined with the finiteness of the $\calH^{m(n - m)}$-measure of the manifold $G$; in fact, the exact $\calH^{m(n - m)}$-measure of $G$ is computed at the end of \cite[3.2.28]{Fe}. The $O(n)$-invariance was precisely the reason why we introduced the measure $\calH^{m(n - m)}_{\pi}$: the metric $d_{\pi}$ is $O(n)$-invariant, so all the corresponding Hausdorff measures are automatically $O(n)$-invariant. Now $\calH^{m(n - m)}_{\pi}$ and $\gamma_{n,m}$ are both $O(n)$-invariant -- hence uniformly distributed -- measures on $G(n,m)$, and it follows from \cite[Theorem 3.4]{Mat2} that $\gamma_{n,m} = \beta\calH^{m(n - m)}_{\pi}$ for some finite constant $\beta > 0$. We infer that \eqref{form30} gives \eqref{form1}.
\end{proof}


\begin{thebibliography}{SSS94}
\bibitem{Bo} \textsc{V. I. Bogachev}: \emph{Measure Theory}, Springer, 2006
\bibitem{FH} \textsc{K. J. Falconer and J. Howroyd}: \emph{Packing dimensions of projections and dimension profiles}, Math. Proc. Cambridge Philos. Soc. \textbf{121}, Issue 2 (1997), pp. 269--286
\bibitem{Fe} \textsc{H. Federer}: \emph{Geometric Measure Theory}, Springer, 1969
\bibitem{Jar} \textsc{M. J\"arvenp\"a\"a}: \emph{On the Upper Minkowski Dimension, the Packing Dimension, and Orthogonal Projections}, Ann. Acad. Sci. Fenn. Ser. A I Math. Dissertationes \textbf{99} (1994)
\bibitem{JJK} \textsc{E. J\"arvenp\"a\"a, M. J\"arvenp\"a\"a, T. Keleti}: \emph{Hausdorff dimension and non-degenerate families of projections}, arXiv: 1203.5296v1
\bibitem{JJLL} \textsc{E. J\"arvenp\"a\"a, M. J\"arvenp\"a\"a, F. Ledrappier and M. Leikas}: \emph{One-dimensional families of projections}, Nonlinearity \textbf{21} (2008), pp. 453--463
\bibitem{JP} \textsc{H. Joyce and D. Preiss}: \emph{On the existence of subsets of finite positive packing measure}, Mathematika \textbf{42} (1995), pp. 15--24
\bibitem{Ka} \textsc{R. Kaufman}: \emph{On Hausdorff dimension of projections}, Mathematika \textbf{15} (1968), pp. 153--155
\bibitem{Lu} \textsc{A. Lubin}: \emph{Extensions of measures and the von Neumann selection theorem}, Proc. Amer. Math. Soc. \textbf{43}(1) (1974), pp. 118--122
\bibitem{Mat1} \textsc{P. Mattila}: \emph{Orthogonal Projections, Riesz Capacities, and Minkowski Content}, Indiana Univ. Math. J. \textbf{39}, Issue 1 (1990), pp. 185--198
\bibitem{Mat2} \textsc{P. Mattila}: \emph{Geometry of sets and measures in Euclidean spaces}, Cambridge University Press, 1995
\bibitem{Or} \textsc{T. Orponen}: \emph{On the Packing Dimension and Category of Exceptional Sets of Orthogonal Projections}, arXiv:1204.2121
\bibitem{Wh} \textsc{H. Whitney}: \emph{Geometric Integration Theory}, Princeton University Press, 1957
\end{thebibliography}
\end{document}